\documentclass{amsart}

\usepackage{amsmath}
\usepackage{amsfonts}
\usepackage{amssymb}
\newcommand{\R}{{\mathbb R}}

\usepackage{enumerate}
\usepackage{xcolor}

\usepackage{graphicx}
\usepackage{float}
\usepackage{subfigure}
\usepackage{tikz}

\newcommand{\Hmm}[1]{\leavevmode{\marginpar{\tiny%
$\hbox to 0mm{\hspace*{-0.5mm}$\leftarrow$\hss}%
\vcenter{\vrule depth 0.1mm height 0.1mm width \the\marginparwidth}%
\hbox to
0mm{\hss$\rightarrow$\hspace*{-0.5mm}}$\\\relax\raggedright #1}}}

\newtheorem{theorem}{Theorem}[section]
\newtheorem{lemma}[theorem]{Lemma}
\newtheorem{corollary}[theorem]{Corollary}

\newtheorem{remark}[theorem]{Remark}

\newtheorem{proposition}[theorem]{Proposition}

\numberwithin{equation}{section}
\usepackage{hyperref}
\usepackage{cleveref}

\begin{document}

\title[Isocapacitary constants for the $p$-Laplacian on compact manifolds]{Isocapacitary constants for the $p$-Laplacian on compact manifolds}

\author{Lili Wang}
\address{Lili Wang: School of Mathematics and Statistics, Key Laboratory of Analytical Mathematics and Applications (Ministry of Education), Fujian Key Laboratory of Analytical Mathematics and Applications (FJKLAMA), Fujian Normal University, 350117 Fuzhou, China.}
\email{\href{mailto:liliwang@fjnu.edu.cn}{liliwang@fjnu.edu.cn}}

\author{Tao Wang}
\address{Tao Wang: Beijing International Center for Mathematical Research, Peking University, 100871, Beijing, China}
\email{\href{mailto:taowang25@pku.edu.cn}{taowang25@pku.edu.cn}}

\date{\today}
\subjclass[2020]{35P15, 35P30, 58J50}
\keywords{Isocapacitary constants, Sobolev constants, Steklov problem, $p$-Laplacian}

\begin{abstract}
    In this paper, we introduce Steklov and Neumann isocapacitary constants for the $p$-Laplacian on compact manifolds. These constants yield two-sided bounds for the $(p,\alpha)$-Sobolev constants, which degenerate to upper and lower bounds for the first nontrivial Steklov and Neumann eigenvalues of the $p$-Laplacian when $\alpha= 1$.
\end{abstract}

\maketitle

\section{Introduction}

Let $(M, g)$ be a $n$-dimensional compact, connected, Riemannian manifold with smooth boundary $\partial M$. The three prototypical eigenvalue problems---Dirichlet, Neumann, and Steklov---on such manifolds have discrete spectra, and quantitative eigenvalue estimates are of interest in spectral geometry.
In \cite{Cheeger1970}, Cheeger discovered a close relation between the first nontrivial eigenvalue of the Laplace-Beltrami operator on a closed manifold and the isoperimetric constant, called the Cheeger constant. Estimates of this type are called Cheeger estimates. Mathematicians have developed numerous generalizations of Cheeger estimates; we refer to \cite{AM1985, DK1986, Dodziuk1984, Escobar1999, HM2020, Jammes2015} and the references therein.

Maz'ya \cite{M1964, M2009, M1962} introduced the isocapacitary constant and used it to estimate the first eigenvalue of the  Dirichlet Laplacian. He proved that if $\Omega$ is a subdomain of a $n$-dimensional Riemannian manifold $M$ and $\lambda_1(\Omega)$ denotes the first Dirichlet eigenvalue of the Laplacian, then
\[
 \frac14\Gamma(\Omega) \leq \lambda_1(\Omega) \leq \Gamma(\Omega),
\]
where 
\[
 \Gamma(\Omega):= \inf_{F \subset \subset \Omega}\frac{\mathrm{Cap}(F, \Omega)}{\mathrm{Vol}(F)},
\]
with
\[
 \mathrm{Cap}(F, \Omega):= \inf\left\{\int_{\Omega}|\nabla u|^2 \ \mathrm{d}x: u \in C^{\infty}_0(\Omega), u \geq 1 \text{ on }F\right\}.
\]
Similar estimates hold for the first Dirichlet eigenvalue of the $p$-Laplacian, see \cite{Gri1999, Mazya2011}. For the use of isocapacitary constants to estimate eigenvalues in other settings, we refer to \cite{HMW2024, HS2025, HW2025}. Another interesting topic regarding capacity is the isocapacitary inequality, for which we refer to \cite{DGMM2021, Mukoseeva2023}.

This work is motivated by the pioneering work of V. G. Maz'ya on capacity theory, isoperimetric inequalities, and Sobolev inequalities, as well as subsequent developments of these estimates by A. Grigor'yan \cite{Gri1999}. Inspired particularly by the approach in \cite[Chapter 9]{Li2012}, we investigate the isocapacitary and Sobolev constants related to the $p$-Laplacian on the manifold $M$.

Given a measurable subset $U \subset \partial M$, we denote the area of $U$ by 
$$\mathrm{Area}(U)=\int_U 1 \ \mathrm{d}S,$$
where $\mathrm{d}S$ is the $(n-1)$-dimensional Hausdorff measure on $\partial M$ induced by the metric $g$. Then for any $p > 1$ and $\alpha>0$, we define the Steklov $(p, \alpha)$-isocapacitary constant $\Gamma_{p,\alpha}^S(M)$ via
\[
\Gamma_{p,\alpha}^S(M):=\inf\limits_{A,B\subset\partial M}\frac{\mathrm{Cap}_p(A,B,M)}{\left(\min\{\mathrm{Area}(A),\mathrm{Area}(B)\}\right)^\frac{1}{\alpha}},
\]
where the infimum is taken over all compact subsets $A, B$ of $\partial M$, and $\mathrm{Cap}_p(A, B, M)$ is the $p$-capacity between $A$ and $B$, defined by equation \eqref{equ:definition_of_capacity}. We also define the Steklov $(p, \alpha)$-Sobolev constant $\mathrm{SS}_{p, \alpha}(M)$ via $$\mathrm{SS}_{p,\alpha}(M):=\inf_{\substack{f \in C^\infty(M) \\ f \not\equiv const}}\frac{\int_M |\nabla f|^p\ \mathrm{d}V}{\left(\min\limits_{c\in\mathbb{R}}\int_{\partial M}|f-c|^{p\alpha}\ \mathrm{d}S\right)^\frac{1}{\alpha}}.$$ 
We now state our main theorem.
\begin{theorem}\label{mt-1}
    Assume that $(M,g)$ is a compact Riemannian manifold with smooth boundary $\partial M$. Then for any $\alpha\geq \frac{1}{p}$, we have  
    \[\frac{(p-1)^{p-1}}{2^\frac{1}{\alpha}p^p}\Gamma_{p,\alpha}^S(M)\leq \mathrm{SS}_{p,\alpha}(M)\leq 2^{\frac{p\alpha-1}{\alpha}}\Gamma_{p,\alpha}^S(M).\]
\end{theorem}

As an application of Theorem \ref{mt-1}, we estimate the first nontrivial eigenvalue of the $p$-Laplacian Steklov problem.

Following \cite{Provenzano2022}(see \cite{Pinasco2007} for a different variational characterization of the eigenvalue), we consider the $p$-Laplacian Steklov problem: for $p > 1$,  
\begin{equation}\label{equ:steklov_problem}
 \begin{cases}
  \Delta_p u = 0, \quad &\text{ in } M, \\
  |\nabla u|^{p-2}\frac{\partial u}{\partial \nu} = \sigma |u|^{p-2}u, &\text{ on } \partial M,
 \end{cases}
\end{equation}
where $\nu$ is the outward unit normal on the boundary $\partial M$. 
The problem is understood in the weak sense, that is, a couple $(u, \sigma) \in W^{1, p}(M) \times \R$ is a weak solution to \eqref{equ:steklov_problem} if and only if
\[
 \int_M\langle |\nabla u|^{p-2}\nabla u, \nabla \varphi\rangle \ \mathrm{d}V = \sigma\int_{\partial M}|u|^{p-2}u\varphi \ \mathrm{d}S, \quad \forall \ \varphi \in W^{1, p}(M).
\]
 We denote by $\sigma_{1, p}(M)$ the first nontrivial variational eigenvalue of the above problem. Set $\Gamma_p^S(M) = \Gamma_{p, 1}^S(M)$. Then we have the following estimate.
\begin{corollary}\label{mt-2}
    Assume that $(M,g)$ is a compact Riemannian manifold with smooth boundary $\partial M$. Then 
    $$\frac{(p-1)^{p-1}}{2p^p}\Gamma_p^S(M)\leq \sigma_{1, p}(M)\leq 2^{p-1}\Gamma_p^S(M).$$
\end{corollary}

We now introduce the corresponding Neumann constant. For any $p > 1$ and $\alpha>0$, the Neumann $(p, \alpha)$-isocapacitary constant $\Gamma_{p,\alpha}^N(M)$ is defined as
\begin{equation}\label{definition-of-N-isocap}
\Gamma_{p,\alpha}^N(M):=\inf\limits_{A,B\subset M}\frac{\mathrm{Cap}_p(A,B,M)}{(\min\{\mathrm{Vol}(A),\mathrm{Vol}(B)\})^\frac{1}{\alpha}},    
\end{equation}
where the infimum is taken over all compact subsets $A, B$ of $M$. Recall that the Neumann $(p, \alpha)$-Sobolev constant $\mathrm{NS}_{p, \alpha}(M)$ is defined by $$\mathrm{NS}_{p,\alpha}(M):=\inf_{\substack{f \in C^\infty(M) \\ f \not\equiv const}}\frac{\int_M |\nabla f|^p\ \mathrm{d}V}{\left(\min\limits_{c\in\mathbb{R}}\int_{M}|f-c|^{p\alpha}\ \mathrm{d}V\right)^\frac{1}{\alpha}}.$$
We have the following theorem.
\begin{theorem}\label{mt-3}
    Assume that $(M,g)$ is a compact Riemannian manifold with smooth boundary $\partial M$. Then for any $\alpha\geq \frac{1}{p}$, we have  
    \[\frac{(p-1)^{p-1}}{2^\frac{1}{\alpha}p^p}\Gamma_{p,\alpha}^N(M)\leq \mathrm{NS}_{p,\alpha}(M)\leq 2^{\frac{p\alpha-1}{\alpha}}\Gamma_{p,\alpha}^N(M).\]
\end{theorem}
As a corollary, we obtain a Maz’ya-type estimate for the first nontrivial Neumann eigenvalue of the $p$-Laplacian. For any $p>1$, the Neumann problem of the $p$-Laplacian is defined as 
\begin{equation}\label{Neumann-eigen-equation}
\begin{cases}
 -\Delta_pu=\mu|u|^{p-2}u,\quad &\text{ in } \ \ M,\\
 \frac{\partial u}{\partial \nu}=0, &\text{ on } \ \partial M.
\end{cases}    
\end{equation}
Let $\mu_{1,p}(M)$ be the first nontrivial eigenvalue of the Neumann problem \eqref{Neumann-eigen-equation}. 
Note that in case of $\partial M=\emptyset,$ the above eigenvalue problem is reduced to the $p$-Laplacian eigenvalue of a closed manifold. Setting $\Gamma_p^N(M) = \Gamma_{p,1}^N(M)$, we have the following estimate. 
\begin{corollary}\label{mt-4}
    Assume that $(M,g)$ is a compact Riemannian manifold with smooth boundary $\partial M$, then 
    \[\frac{(p-1)^{p-1}}{2p^p}\Gamma_p^N(M)\leq \mu_{1, p}(M)\leq 2^{p-1}\Gamma_p^N(M).\]
\end{corollary}
\begin{remark}
\
\begin{enumerate}
\item In case that $M$ is a closed manifold, the estimate of above theorem yields the $p$-isocapacitary estimate for the $p$-Laplacian of a closed Riemannian manifold.  For a closed Riemannian manifold $M$,
\[
\frac{(p-1)^{p-1}}{2p^p}\Gamma_p^N(M)\leq \mu_{1, p}(M)\leq 2^{p-1}\Gamma_p^N(M),
\]

where $\Gamma_p^N(M)$ is given by \eqref{definition-of-N-isocap} with $\alpha=1$. 

\item  For a closed $n$-dimensional Riemannian manifold $M$ with $\operatorname{Ric} \geq -(n-1)$, Matei \cite{Matei2000} proved Cheeger-type estimate for the first nontrival eigenvalue $\mu_{1, p}(M)$ of the $p$-Laplacian:
\[
\left(\frac{h(M)}{p}\right)^p\leq \mu_{1,p}(M)\leq c(n,p)\left(h(M)+h(M)^p\right),
\]
where $c(n,p)$ is a constant depending on $n$ and $p$. The Cheeger constant $h(M)$ is defined by
\[
h(M):=\inf\limits_{S}\frac{\mathrm{Area}(S)}{\min\{\mathrm{Vol}(M_1), \mathrm{Vol}(M_2)\}},
\]
 where the infimum is taken over all closed hypersurfaces $S$ that divide $M$ into two open submanifolds $M_1$
and $M_2$. 
In contrast, the estimate we prove yields matching-order upper and lower bounds, and holds without the Ricci curvature lower bound required.
\end{enumerate}
\end{remark}

The paper is organized as follows. In Section \ref{sec2}, we set the notation and recall some preliminary results about $p$-capacity, the key lemmas needed later are also proved here. In Section \ref{sec3}, we discuss the Steklov cases and prove \Cref{mt-1} and Corollary \ref{mt-2}. In Section \ref{sec4}, we briefly study the Neumann cases and prove \Cref{mt-3} and Corollary \ref{mt-4}.

\section{Preliminaries}\label{sec2}

 In this section, we introduce the definition and basic properties of $p$-capacity. Our setup differs slightly from that in \cite[Chapter 2]{Mazya2011}, and we provide proofs for all discrepancies for the sake of completeness.

 Assume that $(M,g)$ is a compact Riemannian manifold with smooth boundary $\partial M$, and that $A, B\subset M$ are disjoint compact subsets. Consider the following space of some smooth functions,
$$\mathfrak{R}(A,B, M)=\left\{f\in C^{\infty}\left( M\right): \ f\geq 1\text{ on }A,\ f\leq 0\text{ on }B\right\}.$$
For any $p > 1$, the $p$-capacity between $A$ and $B$ relative to $ M$ is defined as 
\begin{equation}\label{equ:definition_of_capacity}
 \mathrm{Cap}_p(A,B, M)=\inf\limits_{u\in\mathfrak{R}(A,B, M)}\int_{M}|\nabla u|^p \ \mathrm{d}V.
\end{equation}
By convention, if $A \cap B \neq \varnothing$, we set $$\mathrm{Cap}_p(A, B, M) = +\infty;$$ if $A = \varnothing$ or $B = \varnothing$, then $$\mathrm{Cap}_p(A, B, M) = 0.$$
\begin{lemma}\label{l-1-1}
    Assume that $A,B$ are disjoint compact subsets of $M$, then 
    $$\mathrm{Cap}_p(A,B, M)=\inf\limits_{u\in\mathfrak{R}^\prime(A,B, M)}\int_{M}|\nabla u|^p \ \mathrm{d}V,$$
    where $\mathfrak{R}^\prime(A,B, M)$ is a space of smooth functions defined by
    $$\mathfrak{R}^\prime(A,B, M)=\left\{f\in C^{\infty}\left(M\right): \ \begin{matrix} f= 1\text{ in a neighborhood of }A,\\ \text{and }f=0\text{ in a neighborhood of }B, \\ 0 \leq f \leq 1.\end{matrix}\right\}.$$
\end{lemma}
\begin{proof}
    Denote by
    $$\mathrm{Cap}_p^\prime(A,B, M)=\inf\limits_{u\in\mathfrak{R}^\prime(A,B, M)}\int_{M}|\nabla u|^p \ \mathrm{d}V.$$
    It is obvious that $\mathfrak{R}^\prime(A,B, M)\subset\mathfrak{R}(A,B, M)$, hence $$\mathrm{Cap}_p(A,B, M)\leq\mathrm{Cap}_p^\prime(A,B, M).$$ On the other hand, for any $\varepsilon>0$, take $f\in\mathfrak{R}(A,B, M)$ such that
    $$\int_{ M}|\nabla f|^p \ \mathrm{d}V \leq \mathrm{Cap}_p(A,B, M)+\varepsilon.$$
Let $\{\lambda_m(t)\}_{m\geq 1}$ denote a sequence of functions in $C^\infty(\mathbb{R})$ satisfying the following conditions:
\begin{enumerate}
    \item $0\leq\lambda_m^\prime(t)\leq 1+m^{-1}$;
    \item $\lambda_m(t)=0$ in a neighborhood of $(-\infty,0]$ and $\lambda_m(t)=1$ in a neighborhood of $[1,\infty)$;
    \item $0\leq\lambda_m(t)\leq 1$, $\forall \ t \in \mathbb{R}$.
\end{enumerate}
Then we have 
$\lambda_m(f(x))\in\mathfrak{R}^\prime(A,B, M)$ and 
\begin{align*}\mathrm{Cap}_p^\prime(A,B, M)\leq\int_{M}|\nabla \lambda_m(f)|^p \ \mathrm{d}V &\leq (1+m^{-1})^p\int_{ M}|\nabla f|^p \ \mathrm{d}V\\
&\leq (1+m^{-1})^p(\mathrm{Cap}_p(A,B, M)+\varepsilon).
\end{align*}
By letting $m\to\infty$ and $\epsilon\to 0$, we have
$$\mathrm{Cap}_p^\prime(A,B, M)\leq\mathrm{Cap}_p(A,B, M).$$ The proof is complete.
\end{proof}

Denote by
$$\Lambda=\left\{\lambda\in C^\infty(\mathbb{R}):\ \begin{matrix}
    \lambda\text{ is non-decreasing. } \lambda(t)=0\text{ for }t\leq 0,\\ \lambda(t)=1\text{ for }t\geq 1
    \text{ and } \mathrm{Supp}(\lambda')\subset (0,1).
\end{matrix}\right\}.$$
The following lemma comes from \cite[p. 144, Lemma 2]{Mazya2011}.
\begin{lemma}\label{l-fun}
    Let $g$ be a non-negative function that is integrable on $[0,1]$. Then
    $$\inf\limits_{\lambda\in\Lambda}\int_0^1(\lambda^\prime)^pg \ \mathrm{d}t=\left(\int_0^1\frac{\mathrm{d}t}{g^{1/(p-1)}}\right)^{1-p}.$$
\end{lemma}

Following the approach of \cite[p. 144, Lemma 1]{Mazya2011}, we provide a new representation for $\mathrm{Cap}_p(A,B, M)$.
\begin{lemma}\label{l-cap}
    For any two disjoint compact subsets $A,B\subset M$, we have 
    $$\mathrm{Cap}_p(A,B, M)=\inf\limits_{u\in\mathfrak{R}(A,B, M)}\left\{\int_0^1\frac{\mathrm{d}t}{\left(\int_{M_u^t}|\nabla u|^{p-1} \ \mathrm{d}S_t\right)^{1/(p-1)}}\right\}^{1-p}$$
    where $M_u^t=\{x\in M: |u(x)|=t\}$ for any $t\in\mathbb{R}$ and $\mathrm{d}S_t$ is the $(n-1)$-dimensional Huasdorff measure on $M_u^t$ induced by $\mathrm{d}V$.
\end{lemma}
\begin{proof}
    For $u\in\mathfrak{R}(A,B, M)$ and any $\lambda\in\Lambda$, we have
    \begin{align*}
         \int_{ M}|\nabla \lambda(u)|^p \ \mathrm{d}V&=\int_{ M}(\lambda^\prime(u)|\nabla u|)^p \ \mathrm{d}V
         =\int_0^1 \mathrm{d}t\int_{M_u^t}(\lambda^\prime(u))^p|\nabla u|^{p-1} \ \mathrm{d}S_t\\
         &=\int_0^1 (\lambda^\prime(t))^p\left(\int_{M_u^t}|\nabla u|^{p-1} \ \mathrm{d}S_t\right) \mathrm{d}t.
    \end{align*}
    Together with Lemma \ref{l-fun}, we obtain
    \begin{align*}
     \mathrm{Cap}_p(A,B, M)&\leq\inf\limits_{\substack{u\in\mathfrak{R}(A,B, M)\\ \lambda\in\Lambda}}\int_{ M}|\nabla \lambda(u)|^p \ \mathrm{d}V \\
     &=\inf\limits_{\substack{u\in\mathfrak{R}(A,B, M)\\ \lambda\in\Lambda}}\int_0^1 (\lambda^\prime(t))^p\left(\int_{M_u^t}|\nabla u|^{p-1} \ \mathrm{d}S_t\right)\mathrm{d}t\\
     &\leq \inf\limits_{u\in\mathfrak{R}(A,B, M)}\left\{\int_0^1\frac{\mathrm{d}t}{\left(\int_{M_u^t}|\nabla u|^{p-1} \ \mathrm{d}S_t\right)^{1/(p-1)}}\right\}^{1-p}.
    \end{align*}
    On the other hand, for any $u\in\mathfrak{R}(A,B, M)$, we have
    \begin{align*}
        \int_{ M}|\nabla u|^p \ \mathrm{d}V&=\int_{0}^{\infty} \mathrm{d}t\int_{ M_u^t}|\nabla u|^{p-1} \ \mathrm{d}S_t \geq\int_0^1 \mathrm{d}t\int_{ M_u^t}|\nabla u|^{p-1} \ \mathrm{d}S_t\\
        &\geq\left\{\int_0^1\frac{\mathrm{d}t}{\left(\int_{ M_u^t}|\nabla u|^{p-1} \ \mathrm{d}S_t\right)^{1/(p-1)}}\right\}^{1-p}.
    \end{align*}
    It follows that
    $$\mathrm{Cap}_p(A,B, M)\geq\inf\limits_{u\in\mathfrak{R}(A,B, M)}\left\{\int_0^1\frac{\mathrm{d}t}{\left(\int_{ M_u^t}|\nabla u|^{p-1} \ \mathrm{d}S_t\right)^{1/(p-1)}}\right\}^{1-p}.$$
    The proof is complete.
\end{proof}

For any $t\in\mathbb{R}$ and $u\in C^\infty(M)$, set
$$ \mathcal{N}_u^{\geq t}=\{x\in M: u(x) \geq t\}\text{ and }\mathcal{N}_u^{\leq t}=\{x\in M: u(x) \leq t\}.$$
Suppose that $u \in C^{\infty}(M)$ satisfies
\[
 T=\sup\left\{t>0: \mathrm{Cap}_p\left(\mathcal{N}_u^{\geq t}, \mathcal{N}_u^{\leq 0}, M\right)>0\right\}>0.
\]
We claim that for $0<t<T$,
\[
 \psi(t):=\int_0^t \frac{\mathrm{d}\tau}{[g(\tau)]^{1/(p-1)}}<\infty,
\]
where 
\[
 g(\tau):= \int_{M_u^{\tau}} |\nabla u|^{p-1} \ \mathrm{d}S_{\tau}.
\]
Indeed, let 
 $$v(x)=\begin{cases}
     t^{-2}u(x)^2&\text{ if }u(x)\geq 0;\\
     -t^{-2}u(x)^2&\text{ if }u(x)\leq 0.
 \end{cases}$$
Since $v \in \mathfrak{R}(\mathcal{N}_u^{\geq t}, \mathcal{N}_u^{\leq 0}, M)$, by Lemma \ref{l-cap}, we have 
\[
 \int_0^1 \left(\int_{M_v^{\tau}}|\nabla v|^{p-1} \ \mathrm{d}S_{\tau}\right)^{1/(1-p)}\mathrm{d}\tau \leq [\mathrm{Cap}_p(\mathcal{N}_u^{\geq t}, \mathcal{N}_u^{\leq 0}, M)]^{1/(1-p)}<\infty.
\]
The claim follows from the fact that 
\[
 \int_0^t \frac{\mathrm{d}\tau}{[g(\tau)]^{1/(p-1)}} = \int_0^1 \left(\int_{M_v^{\tau}}|\nabla v|^{p-1} \ \mathrm{d}S_{\tau}\right)^{1/(1-p)}\mathrm{d}\tau.
\]
It follows that $g(\tau) < \infty$ for almost $0 < t < T$ and the function $\psi(t)$ is strictly monotonic. Consequently, on the interval $[0, \psi(T))$ the function $t(\psi)$, which is the inverse of $\psi(t)$, exists. Similarly to the proof of \cite[p. 153, Lemma]{Mazya2011}, we have the following lemma.
\begin{lemma}\label{l-function}
    Assume that $(M,g)$ is a compact Riemannian manifold with smooth boundary $\partial M$ and $u\in C^\infty\left(M\right)$ such that $T > 0$.
    With the same notations as above, then the function $t(\psi)$ is absolute continuous on segment $[0,\psi(T-\delta)]$ for any $0<\delta<T$, and 
    $$\int_{M}|\nabla u|^p \ \mathrm{d}V\geq\int_{0}^{\psi(T)}t^\prime(\psi)^p \ \mathrm{d}\psi.$$
\end{lemma}
Now we prove the following inequality of capacity.
\begin{proposition}\label{p-cap}
    Assume that $(M,g)$ is a compact Riemannian manifold with smooth boundary $\partial M$ and $u\in C^\infty\left(M\right)$, then
    $$\int_0^{\infty} \mathrm{Cap}_p\left(\mathcal{N}_u^{\geq t}, \mathcal{N}_u^{\leq 0},  M\right) \mathrm{d}(t^p)\leq \frac{p^p}{(p-1)^{p-1}}\int_{M} |\nabla u|^p \ \mathrm{d}V.$$
\end{proposition}
\begin{proof}
 We only need to consider the case when $T > 0$. By changing of variable, we have
 \begin{align}\label{e-1}
     \int_0^{\infty} \mathrm{Cap}_p\left(\mathcal{N}_u^{\geq t}, \mathcal{N}_u^{\leq 0}, M\right) \mathrm{d}(t^p)&=\int_0^T \mathrm{Cap}_p\left(\mathcal{N}_u^{\geq t}, \mathcal{N}_u^{\leq 0}, M\right) \mathrm{d}(t^p)\\
     &=\int_0^{\psi(T)}\mathrm{Cap}_p\left(\mathcal{N}_u^{\geq t(\psi)}, \mathcal{N}_u^{\leq 0}, M\right) \mathrm{d}(t(\psi)^p).\nonumber
 \end{align}
 For any $t>0$, set 
 $$v(x)=\begin{cases}
     t^{-2}u(x)^2&\text{ if }u(x)\geq 0;\\
     -t^{-2}u(x)^2&\text{ if }u(x)\leq 0.
 \end{cases}$$
 Then we have $v\in\mathfrak{R}\left(\mathcal{N}_u^{\geq t}, \mathcal{N}_u^{\leq 0}, M\right)$ and 
 \begin{align}\label{e-2}
   &\mathrel{\phantom{=}}\int_0^1\left(\int_{M_v^{\tau}}|\nabla v|^{p-1} \ \mathrm{d}S_{\tau}\right)^{1/(1-p)}\mathrm{d}\tau \nonumber \\
   & = \int_0^t\left(\int_{M_u^{\xi}}(2t^{-2}\xi)^{p-1}|\nabla u|^{p-1} \ \mathrm{d}S_{\xi}\right)^{1/(1-p)}2\xi t^{-2} \ \mathrm{d}\xi  \\
   & = \psi(t). \nonumber
 \end{align}
 Together with Lemma \ref{l-cap} and equation \eqref{e-2}, it follows that
 \begin{align*}
     \mathrm{Cap}_p\left(\mathcal{N}_u^{\geq t}, \mathcal{N}_u^{\leq 0}, M\right)\leq \left(\frac{1}{\psi(t)}\right)^{p-1}.
 \end{align*}
 Combining with equation \eqref{e-1}, we have 
 \begin{align}\label{e-3}
     \int_0^{\infty} \mathrm{Cap}_p\left(\mathcal{N}_u^{\geq t}, \mathcal{N}_u^{\leq 0}, M\right) \mathrm{d}(t^p) \leq p\int_0^{\psi(T)}\left(\frac{t(\psi)}{\psi}\right)^{p-1}t^\prime(\psi) \ \mathrm{d}\psi.
 \end{align}
 The Hardy inequality implies that
 \begin{align}\label{e-4}
     \int_0^{\psi(T)}\frac{t(\psi)^p}{\psi^p} \ \mathrm{d}\psi\leq \left(\frac{p}{p-1}\right)^p\int_0^{\psi(T)}t^\prime(\psi)^p \ \mathrm{d}\psi.
 \end{align}
 From \eqref{e-3}, \eqref{e-4}, Lemma \ref{l-function} and H\"older's inequality, we have 
 \begin{align*}
   \int_0^{\infty} \mathrm{Cap}_p\left(\mathcal{N}_u^{\geq t}, \mathcal{N}_u^{\leq 0}, M\right) \mathrm{d}(t^p) 
   &\leq p\int_0^{\psi(T)}\left(\frac{t(\psi)}{\psi}\right)^{p-1}t^\prime(\psi) \ \mathrm{d}\psi \\
   &\leq p\left(\int_0^{\psi(T)}\frac{t(\psi)^p}{\psi^p} \ \mathrm{d}\psi\right)^{\frac{p-1}{p}}\left(\int_0^{\psi(T)}t^\prime(\psi)^p \ \mathrm{d}\psi\right)^{\frac{1}{p}} \\
   &\leq \frac{p^p}{(p-1)^{p-1}}\int_0^{\psi(T)}t^\prime(\psi)^p \ \mathrm{d}\psi \\
   &\leq \frac{p^p}{(p-1)^{p-1}}\int_{M} |\nabla u|^p \ \mathrm{d}V,
 \end{align*}
 which completes the proof.
\end{proof}

\section{\texorpdfstring{Two-sided estimates of Steklov $(p,\alpha)$-Sobolev constants and the first eigenvalue via isocapacity}{Two-sided estimates of Steklov (p,alpha)-Sobolev constants and the first eigenvalue via isocapacity}}\label{sec3}

This section treats the Steklov case. We prove Theorem \ref{mt-1}, which establishes two-sided estimates for the $(p,\alpha)$-Sobolev constant in terms of the corresponding Steklov $(p,\alpha)$-isocapacity constants.
The bound directly yields the first nontrivial Steklov eigenvalue estimate  stated in Corollary \ref{mt-2}.

Recall that 
\[
\Gamma_{p,\alpha}^S(M):=\inf\limits_{A,B\subset\partial M}\frac{\mathrm{Cap}_p(A,B,M)}{\left(\min\{\mathrm{Area}(A),\mathrm{Area}(B)\}\right)^\frac{1}{\alpha}},
\]
and $$\mathrm{SS}_{p,\alpha}(M):=\inf_{\substack{f \in C^\infty(M) \\ f \not\equiv const}}\frac{\int_M |\nabla f|^p\ \mathrm{d}V}{\left(\min\limits_{c\in\mathbb{R}}\int_{\partial M}|f-c|^{p\alpha}\ \mathrm{d}S\right)^\frac{1}{\alpha}}.$$

\begin{proof}[Proof of Theorem \ref{mt-1}]
 We first show that 
 \[
 \frac{(p-1)^{p-1}}{2^\frac{1}{\alpha}p^p}\Gamma_{p,\alpha}^S(M)\leq \mathrm{SS}_{p,\alpha}(M). 
 \]
 For any $\varepsilon>0$, there exists $\varphi\in C^\infty(M)$ such that 
 \[
 \frac{\int_M |\nabla\varphi|^p \ \mathrm{d}V}{\left(\min\limits_{c\in\mathbb{R}}\int_{\partial M} |\varphi -c|^{p\alpha}\ \mathrm{d}S\right)^\frac{1}{\alpha}}\leq \mathrm{SS}_{p,\alpha}(M)+\varepsilon.
\]
There exists $c_1\in\mathbb{R}$ such that 
\[
\mathrm{Area}\left(\{x\in\partial M:\varphi \leq c_1\}\right)\geq \frac{1}{2}\mathrm{Area}(\partial M)
\]
and
\[
\mathrm{Area}\left(\{x\in\partial M:\varphi \geq c_1\}\right)\geq \frac{1}{2}\mathrm{Area}(\partial M).
\]
Denote $f=\varphi-c_1$. Note that 
\[
\|f\|_{L^{p\alpha}(\partial M)}^{p\alpha}=\|f_+\|_{L^{p\alpha}(\partial M )}^{p\alpha}+\|f_-\|_{L^{p\alpha}(\partial M )}^{p\alpha},
\]
where $f_+=\max\{f, 0\}$ and $f_-=\max\{-f, 0\}$. Without loss of generality, we assume that 
\[
\|f_+\|_{L^{p\alpha}(\partial M )}^{p\alpha}\geq \frac{1}{2}\|f\|_{L^{p\alpha}(\partial M)}^{p\alpha}.
\]
Then we have 
\begin{align*}
&\mathrel{\phantom{=}}\left( \mathrm{SS}_{p,\alpha}(M)+\epsilon\right)  \|f\|_{L^{p\alpha}(\partial M)}^p \\
&\geq  \int_M |\nabla f|^p\ \mathrm{d}V\\
&\geq \frac{(p-1)^{p-1}}{p^p}\int_0^\infty \mathrm{Cap}_p\left(\mathcal{N}_f^{\geq t},\mathcal{N}_f^{\leq 0}, M\right)\ \mathrm{d}(t^p)\\
&\geq \frac{(p-1)^{p-1}}{p^p}\int_0^\infty \mathrm{Cap}_p\left(\mathcal{N}_f^{\geq t}\cap\partial M, \mathcal{N}_f^{\leq 0}\cap\partial M, M\right)\ \mathrm{d}(t^p)\\
&\geq \frac{(p-1)^{p-1}}{p^p}\Gamma_{p,\alpha}^S(M)\int_0^\infty \mathrm{Area}\left(\{x\in\partial M:f\geq t\}\right)^\frac{1}{\alpha}\ \mathrm{d}(t^p).
\end{align*}
Denote 
\[
\mathcal{E}_t:=\{x\in \partial M: f\geq t\}. 
\]
We claim that 
\begin{equation}\label{claim}
\left(\int_0^\infty\mathrm{Area}(\mathcal{E}_t)^\frac{1}{\alpha}\ \mathrm{d}(t^p)\right)^\alpha\geq p\alpha \int_0^\infty t^{p\alpha-1}\mathrm{Area}(\mathcal{E}_t)\ \mathrm{d}t.    
\end{equation}
In fact, for any $s>0$, we have 
\begin{align*}
\frac{\mathrm{d}}{\mathrm{d}s}\left(\int_0^s \mathrm{Area}(\mathcal{E}_t)^\frac{1}{\alpha}\ \mathrm{d}(t^p)\right)^{\alpha}
&=\alpha\left(\int_0^s \mathrm{Area}(\mathcal{E}_t)^\frac{1}{\alpha}\ \mathrm{d}(t^p)\right)^{\alpha-1}ps^{p-1}\mathrm{Area}(\mathcal{E}_s)^\frac{1}{\alpha}\\
&\geq \alpha\left(\int_0^s \mathrm{Area}(\mathcal{E}_s)^\frac{1}{\alpha}\ \mathrm{d}(t^p)\right)^{\alpha-1}ps^{p-1}\mathrm{Area}(\mathcal{E}_s)^\frac{1}{\alpha}\\
&=\alpha s^{p(\alpha-1)}\mathrm{Area}(\mathcal{E}_s)^\frac{\alpha-1}{\alpha} ps^{p-1}\mathrm{Area}(\mathcal{E}_s)^\frac{1}{\alpha}\\
&=\alpha p s^{p\alpha-1}\mathrm{Area}(\mathcal{E}_s)
\end{align*}
and
\[
\frac{\mathrm{d}}{\mathrm{d}s}\left(\alpha\int_0^s t^{p(\alpha-1)}\mathrm{Area}(\mathcal{E}_t)pt^{p-1}\ \mathrm{d}t\right)
=p\alpha s^{p\alpha-1}\mathrm{Area}(\mathcal{E}_s).
\]
The claim follows by integrating from $0$ to $s$ and letting $s \to \infty$.

Hence, by \eqref{claim}  we obtain
\begin{align*}
&\mathrel{\phantom{=}}\left( \mathrm{SS}_{p,\alpha}(M)+\epsilon\right)  \|f\|_{L^{p\alpha}(\partial M)}^p\\
&\geq \frac{(p-1)^{p-1}}{p^p} \Gamma_{p,\alpha}^S(M)\left(\alpha p\int_0^\infty t^{p\alpha-1} \mathrm{Area}(\mathcal{E}_t)\ \mathrm{d}t\right)^\frac{1}{\alpha}\\
& =\frac{(p-1)^{p-1}}{p^p} \Gamma_{p,\alpha}^S(M)
\left(\alpha p\int_0^\infty t^{p\alpha-1} \int_{\{x\in\partial M:f(x)\geq t\}}\ \mathrm{d}S\mathrm{d}t\right)^\frac{1}{\alpha}\\
&=\frac{(p-1)^{p-1}}{p^p} \Gamma_{p,\alpha}^S(M)
\left(\int_{\{x\in\partial M:f(x)\geq 0\}}\ \mathrm{d}S\int_0^{f(x)}\alpha pt^{p\alpha-1}\ \mathrm{d}t\right)^\frac{1}{\alpha}\\
&=\frac{(p-1)^{p-1}}{p^p} \Gamma_{p,\alpha}^S(M)
\left(\int_{\{x\in\partial M:f(x)\geq 0\}}   f(x)^{p\alpha}\ \mathrm{d}S\right)^\frac{1}{\alpha}\\
&\geq \frac{(p-1)^{p-1}}{p^p} \Gamma_{p, \alpha}^S(M) \|f_+\|^p_{L^{p\alpha}(\partial M)}\\
&\geq \frac{(p-1)^{p-1}}{p^p} \Gamma_{p, \alpha}^S(M) \left(\frac{1}{2}\right)^\frac{1}{\alpha}\|f\|^p_{L^{p\alpha}(\partial M)}.
\end{align*}
We end up with 
\[
\mathrm{SS}_{p,\alpha}(M)\geq \frac{(p-1)^{p-1}}{p^p} \left(\frac{1}{2}\right)^\frac{1}{\alpha}\Gamma_{p,\alpha}^S(M)
\]
by letting $\epsilon\to 0$.

Next, we show that 
\[
\mathrm{SS}_{p,\alpha}(M)\leq 2^{\frac{p\alpha-1}{\alpha}}\Gamma_{p,\alpha}^S(M).
\]
For any $\varepsilon>0$, there exist disjoint compact subsets $A, B \subset \partial M$ such that
\[
\frac{\mathrm{Cap}_p(A, B, M)}{\min\{\mathrm{Area}(A),\mathrm{Area}(B)\}^{\frac{1}{\alpha}}} \leq \Gamma_{p, \alpha}^S(M)+\varepsilon.
\]
By Lemma \ref{l-1-1}, we see that there exists a function $g\in C^\infty(M)$ such that
\begin{equation}\label{g-def}
g\equiv 1 \ \text{ on } \ A, \ g\equiv 0 \ \text{ on } \ B    
\end{equation}
and
\begin{equation}\label{palphainequality}
\int_M|\nabla g|^p\ \mathrm{d}V\leq \mathrm{Cap}_p(A,B,M)+\varepsilon\left(\min\{\mathrm{Area}(A),\mathrm{Area}(B)\}^{\frac{1}{\alpha}}\right).
\end{equation}
Take $c_2 \in \mathbb{R}$ such that 
\[
\int_{\partial M}|g-c_2|^{p\alpha}\ \mathrm{d}S = \min_{c \in \mathbb{R}}\int_{\partial M}|g-c|^{p\alpha}\ \mathrm{d}S.
\]
Then we have
   \begin{align*}
       \mathrm{SS}_{p, \alpha}(M)&\leq\frac{\int_{ M}|\nabla (g-c_2)|^p \ \mathrm{d}V}{\left(\int_{\partial M}|g-c_2|^{p\alpha} \ \mathrm{d}S\right)^{\frac{1}{\alpha}}}\\
       &\leq\frac{\mathrm{Cap}_p(A, B, M)+\varepsilon\left(\min\{\mathrm{Area}(A), \mathrm{Area}(B)\}^{\frac{1}{\alpha}}\right)}{\left((1-c_2)^{p\alpha}\cdot \mathrm{Area}(A)+c_2^{p\alpha}\cdot \mathrm{Area}(B)\right)^{\frac{1}{\alpha}}}\\
       &\leq\frac{\mathrm{Cap}_p(A, B, M)+\varepsilon\left(\min\{\mathrm{Area}(A), \mathrm{Area}(B)\}^{\frac{1}{\alpha}}\right)}{\left[((1-c_2)^{p\alpha}+c_2^{p\alpha})\cdot (\min\{\mathrm{Area}(A), \mathrm{Area}(B)\})\right]^{\frac{1}{\alpha}}} \\
       &\leq 2^{\frac{p\alpha-1}{\alpha}}\cdot\frac{\mathrm{Cap}_p(A, B, M)+\varepsilon\left(\min\{\mathrm{Area}(A), \mathrm{Area}(B)\}^{\frac{1}{\alpha}}\right)}{\min\{\mathrm{Area}(A), \mathrm{Area}(B)\}^{\frac{1}{\alpha}}}\\
       &\leq 2^{\frac{p\alpha-1}{\alpha}}(\Gamma_{p, \alpha}^S(M)+ 2\varepsilon).
   \end{align*}
We end up with 
\[
\mathrm{SS}_{p,\alpha}(M)\leq 2^{\frac{p\alpha-1}{\alpha}}\Gamma_{p,\alpha}^S(M)
\]
by letting $\varepsilon \to 0$.
\end{proof}

At the end of this section, we establish the estimate of the first nontrivial Steklov eigenvalue $\sigma_{1, p}(M)$ of the $p$-Laplacian. By Rayleigh quotient characterization, 
\[
\sigma_{1, p}(M) = \inf \left\{ \frac{\int_{M} |\nabla u|^p \ \mathrm{d}V}{\int_{\partial M} |u|^p \ \mathrm{d}S} \ \middle|\ 
u \in C^{\infty}(M),\ \int_{\partial M} |u|^{p-2} u \ \mathrm{d}S = 0,\ u \not\equiv 0 \right\}.
\]
One can check that 
\begin{equation}\label{equ:RQC_For_Steklov}
\sigma_{1, p}(M) = \inf_{\substack{u \in C^\infty(M) \\ u \not\equiv const}}\frac{\int_M |\nabla u|^p \ \mathrm{d}V}{\min\limits_{c \in \mathbb{R}} \int_{\partial M} |u - c|^p \ \mathrm{d}S}.
\end{equation}
\begin{proof}[Proof of Corollary \ref{mt-2}]
 The conclusion follows from equation \eqref{equ:RQC_For_Steklov} and Theorem \ref{mt-1} by letting $\alpha =1$.  
\end{proof}

\section{\texorpdfstring{Two-sided estimates of Neumann $(p,\alpha)$-Sobolev constants and the first eigenvalue via isocapacity}{Two-sided estimates of Neumann (p, alpha)-Sobolev constants and the first eigenvalue via isocapacity}}\label{sec4}

In this section, we treat the Neumann case. Following the same variational scheme as in Section \ref{sec3}, we establish two-sided bounds for the $(p,\alpha)$-Sobolev constant by Neumann $(p,\alpha)$-isocapacity, and derive the first nontrivial Neumann eigenvalue estimate (Corollary \ref{mt-4}).

Recall that 
\[
\Gamma_{p,\alpha}^N(M):=\inf\limits_{A,B\subset M}\frac{\mathrm{Cap}_p(A,B,M)}{\min\{\mathrm{Vol}(A),\mathrm{Vol}(B)\}^\frac{1}{\alpha}},
\]
and $$\mathrm{NS}_{p,\alpha}(M):=\inf_{\substack{f \in C^\infty(M) \\ f \not\equiv const}}\frac{\int_M |\nabla f|^p\ \mathrm{d}V}{\left(\min\limits_{c\in\mathbb{R}}\int_{M}|f-c|^{p\alpha}\ \mathrm{d}V\right)^\frac{1}{\alpha}}.$$

\begin{proof}[Proof of \Cref{mt-3}]
 The proof is the same as that of \Cref{mt-1}, with the area integral replaced by the volume integral and $\partial M$ replaced by $M$. We omit the details here.
\end{proof}

In the end, we complete the proof of Corollary \ref{mt-4}.
\begin{proof}[Proof of Corollary \ref{mt-4}]
 By Rayleigh quotient characterization,
\[
\mu_{1, p}(M) = \inf_{\substack{u \in C^\infty(M) \\ u \not\equiv const}}\frac{\int_M |\nabla u|^p \ \mathrm{d}V}{\min\limits_{c \in \mathbb{R}} \int_{M} |u - c|^p \ \mathrm{d}V}.
\]
Then the conclusion follows from \Cref{mt-3} by letting $\alpha = 1$.
\end{proof}

\section*{Acknowledgments}
The authors thank Prof. Bobo Hua for his helpful suggestions. L. Wang is supported by NSFC, no. 12371052, and the Fujian Alliance of Mathematics, no. 2024SXLMMS01.

\bibliographystyle{plain}
\bibliography{pLaplacian}

\end{document}